\documentclass{article}
\usepackage{amsmath,amsthm,amssymb,amscd}

\newcommand{\ga}{\alpha}
\newcommand{\gb}{\beta}

\newcommand{\gw}{\omega}

\newcommand{\gs}{\sigma}
\newcommand{\eps}{\varepsilon}

\newcommand{\coll}{\mathrm{Coll}}

\newcommand{\cantor}{2^\gw}

\newcommand{\dotxgen}{{\dot x}_{\mathit{gen}}}

\newcommand{\dom}{\mathrm{dom}}

\newcommand{\power}{\mathcal{P}}

\newtheorem{theorem}{Theorem}[section]

\newtheorem{corollary}[theorem]{Corollary}

\newtheorem{proposition}[theorem]{Proposition}

\theoremstyle{definition}
\newtheorem{definition}[theorem]{Definition}
\newtheorem{example}[theorem]{Example}
\newtheorem{question}[theorem]{Question}

\title{Transcendental pairs of generic extensions\footnote{2010 AMS subject classification 03E15, 03E25, 03E35.}}

\author{
Jind{\v r}ich Zapletal\\
University of Florida}

\begin{document}
\maketitle

\begin{abstract}
I prove several consistency results in choiceless set theory regarding chromatic numbers of hypergraphs on Polish groups defined by group equations, and hypergraphs on $\power(\gw)$ modulo finite defined by boolean identities.
\end{abstract}

\section{Introduction}

This paper is a contribution to the study of Borel hypergraphs in choiceless set theory as started in \cite{z:geometric} and related to the concerns of descriptive graph theory \cite{miller:bsl}. I isolate a certain class of $F_\gs$-hypergraphs on $K_\gs$-spaces and show several consistency results of the following kind: it is consistent with ZF+DC that hypergraphs in this class have countable chromatic number while hypergraphs far from this class have uncountable chromatic number. The method is quite flexible and can be used to prove many consistency results which are not related to chromatic numbers. The theorems I state are more motivational than strongest possible; they are centered around certain interesting hypergraph classes which have been unjustly overlooked so far.

The first class of hypergraphs is connected with Polish groups and group equations. For a Polish group $G$ one may consider the hypergraph on $G$ of solutions to a given group equation, or of solutions to one of a given countable set of group equations. These hypergraphs may serve as a tool distinguishing between various Polish groups. As one example, consider the hypergraph $\Delta(G)$ of quadruples which solve the equation $g_0g_1^{-1}g_2g_3^{-1}=1$.  I prove

\begin{theorem}
\label{1theorem}
Let $G$ be a $K_\gs$ Polish group. It is consistent relative to an inaccessible cardinal that ZF+DC holds, the chromatic number of $\Delta(G)$ is countable, yet the chromatic number of $\Delta(S_\infty)$ is uncountable.
\end{theorem}

\noindent This consistency result depends on the fact that $S_\infty$ has (a small strengthening of) the ample generics property. It seems to be difficult to separate chromatic numbers of the hypergraphs $\Delta(G)$ for various other Polish groups, such as $G=S_\infty$ and $G=$the unitary group.

The second class of hypergraphs is connected with the quotient algebra $\power(\gw)$ modulo finite and its Boolean operations. One can define hypergraphs on it using various Boolean identities. I consider two simple examples. For a natural number $n\geq 2$ let $\Gamma_n$ be the hypergraph on $\power(\gw)$ of arity $n$ consisting of $n$-tuples of sets which modulo finite form a partition of $\gw$, and $\Theta_n$ be the hypergraph of arity $n$ on $\power(\gw)$ consisting of sets $d$ of size $n$ such that $\bigcap d=0$ and $\bigcup d=\gw$, both modulo finite. In ZFC, all of these hypergraphs have chromatic number two, as a membership in any given nonprincipal ultrafilter on $\gw$ is a coloring. However, in the choiceless context clear distinctions arise.

\begin{theorem}
\label{2theorem}
It is consistent relative to an inaccessible cardinal that ZF+DC holds, the chromatic number of $\Gamma_3$ is countable, yet the chromatic number of $\Gamma_4$ is uncountable.
\end{theorem} 

\begin{theorem}
\label{3theorem}
It is consistent relative to an inaccessible cardinal that ZF+DC holds, the chromatic number of $\Gamma_4$ is countable, yet the chromatic number of $\Gamma_5$ is uncountable.
\end{theorem}

\begin{theorem}
\label{4theorem}
It is consistent relative to an inaccessible cardinal that ZF+DC holds, the chromatic number of $\Gamma_4$ is countable, yet the chromatic number of $\Theta_4$ is uncountable.
\end{theorem}

\noindent There are many open questions. In particular, I conjecture that it is possible to separate the chromatic numbers of $\Gamma_n$ and $\Gamma_{n+1}$ for any natural number $n\geq 2$, and similarly for $\Theta_n$ and $\Theta_{n+1}$. In general, arities higher than four present problems that I find difficult to overcome.

Architecture of the paper follows the methodology of geometric set theory \cite{z:geometric}. All models for the consistent theories described in the above theorems are generic extensions of the classical choiceless Solovay model \cite[Theorem 26.14]{jech:newset} via a carefully chosen $\gs$-closed Suslin coloring poset. The consistency results are obtained by a precise calibration of amalgamation properties of conditions in the coloring poset, a process reminiscent of certain concerns of model theory. In Section~\ref{pairsection}, I introduce transcendence of pairs of generic extensions, a property weaker than mutual genericity. In Section~\ref{exampleIsection} I provide a number of useful examples of mutually transcendental pairs of generic extensions. In Section~\ref{preservationsection}, I define the notion of transcendental balance for Suslin forcing, and examples of Section~\ref{preservationsection} are used to prove a number of preservation theorems for generic extensions of the Solovay model obtained with transcendentally balanced forcings. Finally, in Section~\ref{exampleIIsection}, I show that many known balanced forcings are transcendentally balanced--for example the posets adding a transcendence basis for a Polish field over a countable subfield. I also build a new supply of coloring posets in arity three and four which are transcendentally balanced. These are in turn used to prove the theorems of this introduction in the beginning of Section~\ref{exampleIIsection}.

Notation of the paper follows \cite{jech:newset}, and in matters of geometric set theory \cite{z:geometric}. In particular, the calculus of virtual conditions in Suslin forcing of Section~\ref{preservationsection} is established in \cite[Chapter 5]{z:geometric}. A \emph{hypergraph} $\Gamma$ on a set $X$ is just a subset of $\power(X)$, its elements are its \emph{hyperedges}. All hypergraphs in this paper have a fixed finite \emph{arity}, meaning that they are subsets of $[X]^n$ for some number $n\geq 2$. A (partial) function $c$ on $X$ is a $\Gamma$-\emph{coloring} if $c$ is not constant on any $\Gamma$-hyperedge. The chromatic number of $\Gamma$ is the smallest cardinal $\kappa$ such that there is a total $\Gamma$-coloring $c\colon X\to\kappa$. In the absence of the axiom of choice, we only discern between various finite values of the chromatic number and then countable and uncountable chromatic number. 

\section{Transcendental pairs of extensions}
\label{pairsection}

The key concept in this paper is a certain perpendicularity notion for pairs of generic extensions, which generalizes mutual genericity.

\begin{definition}
Let $V[G_0], V[G_1]$ be two generic extensions in an ambient generic extension. Say that $V[G_1]$ is \emph{transcendental over} $V[G_0]$ if for every ordinal $\ga$ and every open set $O\subset 2^\ga$ in the model $V[G_1]$, if $2^\ga\cap V\subset O$ then $2^\ga\cap V[G_0]\subset O$. Say that the models $V[G_0]$, $V[G_1]$ are \emph{mutually transcendental} if each of them is transcendental over the other one.
\end{definition}

\noindent Here, the space $2^\ga$ is equipped with the usual compact product topology. For a finite partial function $h$ from $\ga$ to $2$ write $[h]=\{x\in 2^\ga\colon h\subset x\}$ The open set $O\subset 2^\ga$ is then coded in $V[G_1]$ by a set $H\in V[G_1]$ of finite partial functions from $\ga$ to $2$ with the understanding that $O=\bigcup_{h\in H}[h]$, and in this way it is interpreted in $V[G_0][G_1]$. For the general theory of interpretations of topological spaces in generic extensions (unnecessary for this paper) see \cite{z:interpretations}. It is tempting to deal just with the  Cantor space $2^\gw$ instead of its non-metrizable generalizations, but the present definition has a number of small advantages and essentially no disadvantages as compared to the Cantor space treatment. On the other hand, extending the definition to all compact Hausdorff spaces is equivalent to the present form.

I first need to show that the notion of transcendence generalizes mutual genericity and in general behaves well with respect to product forcing. This is the content of the following proposition.

\begin{proposition}
\label{productproposition}
Let $V[G_0], V[G_1]$ be generic extensions such that $V[G_1]$ is transcendental over $V[G_0]$. Let $P_0\in V[G_0]$ and $P_1\in V[G_1]$ be posets. Let $H_0\subset P_0$ and $H_1\subset P_1$ be filters mutually generic over the model $V[G_0][G_1]$. Then $V[G_1][H_1]$ is transcendental over $V[G_0][H_0]$.
\end{proposition}

\begin{proof}
Work in the model $V[G_0, G_1]$ and consider the product forcing $P_0\times P_1$. Let $\ga$ be an ordinal. Let $\langle p_0, p_1\rangle\in P_0\times P_1$ be a condition, let $\tau\in V[G_1]$ be a $P_1$-name for an open set such that $p_1\Vdash 2^\ga\cap V\subset\dot O$, and let $\eta\in V[G_0]$ be a $P_0$-name for an element of $2^\ga$. We need to find a stronger condition in the product which forces $\eta\in\tau$.

First, work in the model $V[G_0]$. Let $A_0=\{h\colon h$ is a finite partial function from $\ga$ to $2$ and $p_0\Vdash\check h\not\subset\eta\}$. By a compactness argument, the set $\bigcup_{h\in A_0}[h]$ must not cover the whole space $2^\ga$; if it did, a finite subset of $Q_0$ would suffice to cover $2^\ga$ and there would be no space left for the point $\eta$. Let $y\in 2^\ga\setminus\bigcup_{h\in A_0}[h]$ be an arbitrary point. Now work in the model $V[G_1]$ and let $A_1=\{h\colon h$ is a finite partial function from $\ga$ to $2$ and there is a condition $q\leq p_0$ such that $q\Vdash [h]\subset\tau\}$; since $p_0\Vdash 2^\ga\cap V\subset\tau$, it must be the case that $2^\ga\cap V\subset \bigcup_{h\in A_1}[h]$. 

Now, use the transcendence of $V[G_1]$ over $V[G_0]$  to argue that $y\in Q$. It follows that there must be a finite partial function $h\in A_1$ such that $h\subset y$. It follows that there must be conditions $p'_0\leq p_0$ and $p'_1\leq p_1$ such that $p'_0\Vdash\check h\subset\eta$ and $p'_1\Vdash [h]\subset\tau$. Then the condition $\langle p'_0, p'_1\rangle$ forces in the product $P_0\times P_1$ that $\eta\in\tau$ as required.
\end{proof}

\begin{corollary}
Mutually generic extensions are mutually transcendental.
\end{corollary}

\begin{proof}
Just let $V[G_0]=V[G_1]=V$ in Proposition~\ref{productproposition}.
\end{proof}

\noindent In the remainder of this section, I isolate several properties of mutually transcendental extensions which will come handy later.

\begin{proposition}
\label{iproposition}
Let $V[G_0], V[G_1]$ be mutually transcendental generic extensions of $V$. Then $V[G_0]\cap V[G_1]=V$.
\end{proposition}

\begin{proof}
It will be enough to show that $(2^\ga\cap V[G_0])\cap (2^\ga\cap V[G_1])=2^\ga\cap V$ holds for every ordinal $\ga$. Let $x\in 2^\ga\cap V[G_0]\setminus V$ be an arbitrary point. The open set $O=2^\ga\setminus \{x\}$ in $V[G_0]$ covers $2^\ga\cap V$. By the mutual transcendence, $2^\ga\cap V[G_1]\subset O$ must hold as well. In particular, $x\notin V[G_1]$ as required.
\end{proof}

\begin{proposition}
Let $V[G_0], V[G_1]$ be mutually transcendental generic extensions of $V$. Let $X_0, X_1$ be Polish spaces and $C\subset X_0\times X_1$ be a $K_\gs$-set. Let $x_0\in X_0\cap V[G_0]$ and $x_1\in X_1\cap V[G_1]$ be points such that $\langle x_0, x_1\rangle\in C$. Then there is a point $x'_0\in X_0\cap V$ such that $\langle x'_0, x_1\rangle\in C$.
\end{proposition}

\begin{proof}
Since $C$ is a countable union of compact sets, there is a compact set $K\subset C$ coded in $V$ such that $\langle x_0, x_1\rangle\in K$. Let $h\colon\cantor\to X_0$ be a continuous function onto the compact projection of the set $K$ to $X_0$. Let $O\subset\cantor$ in $V[G_1]$ be the open set of all points $y\in\cantor$ such that $\langle f(y), x_1\rangle\notin K$. The set $O$ does not cover $\cantor\cap V[G_0]$ since $h^{-1}x_0\cap O=0$. By the mutual transcendence, there must be a point $y\in\cantor\cap V\setminus O$. Let $x'_0=h(y)$ and observe that the point $x'_0\in X_0$ works as desired.
\end{proof}

\begin{corollary}
\label{jcorollary}
Let $X$ be a $K_\gs$ Polish field. Let $p(\bar v_0, \bar v_1)$ be a multivariate polynomial with coefficents in $X$ and variables $\bar v_0, \bar v_1$. Let $V[G_0], V[G_1]$ be mutually transcendental generic extensions of $V$ and let $\bar x_0\in V[G_0]$, $\bar x_1\in V[G_1]$ be strings of elements of $X$ such that $p(\bar x_0, \bar x_1)=0$. Then there is a string $\bar x_0'\in V$ arbitrarily close to $\bar x_0$ such that $p(\bar x_0', \bar x_1)=0$.
\end{corollary}

\begin{proof}
Apply the proposition with the additional insight that the spaces ${X}^n$ for any natural number $n$ are $K_\gs$ and solutions to a given polynomial form a closed set.
\end{proof}

\begin{corollary}
\label{kscorollary}
Let $E$ be a $K_\gs$-equivalence relation on a Polish space $X$. Whenever $V[G_0], V[G_1]$ are mutually transcendental generic extensions of $V$ and $x_0\in X\cap V[G_0]$ and $x_1\in X\cap V[G_1]$ are $E$-related points, then there is a point $x\in X\cap V$ $E$-related to them both.
\end{corollary}

\noindent The last corollary can be generalized to some non-$K_\gs$-equivalence relations as in the following proposition.

\begin{proposition}
\label{metricproposition}
Let $\langle U_n, d_n\colon n\in\gw\rangle$ be a sequence of sets and metrics on each and let $X=\prod_nU_n$. Let $V[G_0], V[G_1]$ be mutually transcendental generic extensions of $V$ and $x_0\in X\cap V[G_0]$ and $x_1\in X\cap V[G_1]$ be points such that $\lim_n d_n(x_0(n), x_1(n))=0$. Then there is a point $x\in X\cap V$ such that $\lim_n d_n(x(n), x_0(n))=0$.
\end{proposition}

\begin{proof}
First argue that for every number $m\in\gw$ there is a point $y\in X\cap V$ such that $\forall n\ d_n(y(n), x_1(n))\leq 2^{-m}$. To see this, fix a number $k\in\gw$ such that for all $n\geq k$, $d_n(x_0(n), x_1(n))<2^{-m-2}$. Let $A=\{ \{u, v\}\colon \exists n\geq k\ u, v\in U_n$ and $d_n(u,v)>2^{-m}\}$, and consider the space $Z$ of all selectors on $A$, which is naturally homeomorphic to $2^A$. In the model $V[G_0]$, let $O=\{z\in Z\colon\exists n\geq k\ \exists v\in U_n\ d_n(x_0(n), v)>2^{-m}$ and $z(x_0(n), v)=v\}$. This is an open subset of the space $Z$. It does not cover $Z\cap V[G_1]$ as in the model $V[G_1]$, one can find a selector $z\in Z$ such that for all $n\geq k$ and all $\{u, v\}\in Z$ with $u, v\in U_n$, $z(u, v)$ is one of the points $u, v$ which is not $d_n$-farther from $x_1(n)$ than the other. It is immediate from the definition of the set $O$ and a triangle inequality argument that $z\notin O$. By a mutual transcendence argument, there is a selector $z'\in Z\cap V$ such that $z'\notin O$ holds. 

Work in $V$. For each number $n\geq k$, let $B_n=\{u\in U_n\colon\forall v\in Y_n\ d_n(u, v)>2^{-m}\to z'(u, v)=u\}$. The set $B_n$ contains $x_0(n)$ by the choice of the selector $z'$. Moreover, for any two elements $u, v\in B_n$, $d_n(u, v)\leq 2^{-m}$ must hold: in the opposite case, the selector $z'$ could not choose one element from the pair $\{u, v\}$ without contradicting the definition of the set $B_n$. Now consider any point $y\in X$ such that for all $n<k$, $y(n)=x_0(k)$ and for all $n\geq k$ $y(n)\in B_n$. Then $\forall n\ d_n(y(n), x_1(n))\leq 2^{-m}$ as desired.

Now, let $C=\gw\times (X\cap V)$ and consider the set $B\subset C$ of all pairs $\langle m, y\rangle\in A$ such that $\limsup_nd_n(y(n), x_1(n))\leq 2^{-m}$. As written, the set belongs to $V[G_1]$; however, it also belongs to $V[G_0]$ since replacing $x_1$ in its definition with $x_0$ results in the same set by the initial assumptions on $x_0, x_1$. By Proposition~\ref{iproposition}, $B\in V$ holds. By the work in the previous paragraph, for each $m\in\gw$ $B$ contains some element whose first coordinate is $m$. Thus, in $V$ there exists a sequence $\langle y_m\colon m\in\gw\rangle$ such that $\forall m\ \langle m, y_m\rangle\in B$. By a Mostowski absoluteness argument, there must be in $V$ a point $x$ such that for all $m\in\gw$, $\limsup_n(y_m(n), x(n))\leq 2^{-m}$, since such a point, namely $x_0$, exists in $V[G_0]$. A triangular inequality argument then shows that $\lim_n(x(n), x_0(n))=0$ as desired. 
\end{proof}

\noindent I do not know whether further generalizations are possible. In particular, the following is open:

\begin{question}
Let $E$ be a pinned Borel equivalence relation on a Polish space $X$. Let $V[G_0], V[G_1]$ be mutually transcendental generic extensions of $V$ and $x_0\in X\cap V[G_0]$ and $x_1\in X\cap V[G_1]$ are $E$-related points. Must there be a point $x\in X\cap V$ $E$-related to them both?
\end{question}

\section{Examples I}
\label{exampleIsection}

In this section, I provide several interesting pairs of mutually transcendental pairs of generic extensions. To set up the notation, for a Polish space $X$, write $P_X$ for the Cohen poset of nonempty open subsets of $X$ ordered by inclusion, with $\dotxgen$ being its name for a generic element of $X$. If $f\colon X\to Y$ is a continuous open map then $P_X$ forces $f(\dotxgen)\in Y$ to be a point generic for $P_Y$ \cite[Proposition 3.1.1]{z:geometric}. The first definition and proposition deal with Cohen elements of Polish spaces.

\begin{definition}
\label{adefinition}
Let $X, Y_0, Y_1$ be compact Polish spaces and $f_0\colon X\to Y_0$ and $f_1\colon X\to Y_1$ be continuous open maps. Say that $f_1$ is \emph{transcendental over} $f_0$ if for every nonempty open set $O\subset X$ there is a point $y_0\in Y_0$ such that the set $f_1''(f_0^{-1}\{y_0\}\cap O)\subset Y_1$ has nonempty interior for every nonempty open set $O\subset X$.
\end{definition}

\begin{proposition}
Suppose that $X, Y_0, Y_1$ are Polish spaces, $X$ is compact, and $f_0\colon X\to Y_0$ and $f_1\colon X\to Y_1$ are continuous open maps. The following are equivalent:

\begin{enumerate}
\item  $f_1$ is transcendental over $f_0$;
\item $P_X$ forces $V[f_1(\dotxgen)]$ to be transcendental over $V[f_0(\dotxgen)]$.
\end{enumerate}
\end{proposition}

\begin{proof}
To show that (1) implies (2), let $\ga$ be an ordinal, let $\eta$ be a $P_{Y_0}$-name for an element of $2^\ga$, and let $\tau$ be a $P_{Y_1}$-name for an open subset of $2^\ga$ which is forced to contain $V\cap 2^\ga$ as a subset. Let $O\subset X$ be a nonempty open set. To prove (2), I must find a strengthening $O'\subset O$ such that $O'\Vdash\eta/f_0(\dotxgen)\in\tau/f_1(\dotxgen)$.

To this end, let $y_0\in Y_0$ be a point such that the set $f_1''(f_0^{-1}\{y_0\}\cap O)$ has nonempty interior, and let $O_1\subset Y_1$ denote that interior. Use the initial assumption on $\tau$ to find, for each $z\in 2^\ga$, a condition $O_{1z}\subset O_1$ and a finite partial map $h_z\colon\ga\to 2$ such that $h_z\subset z$ and $O_{1z}\Vdash [h_z]\subset\tau$. Use a compactness argument to find a finite set $a\subset 2^\ga$ such that $2^\ga=\bigcup_{z\in a}[h_z]$. The set $O_0=\bigcap_{z\in a}f''_0(O\cap f^{-1}_1O_{1z}\subset Y_0$ is nonempty as it contains $y_0$, and it is open as the maps $f_0, f_1$ are continuous and open. Let $O'_0\subset O_0$ be a condition which decides the value $\eta(\check\gb)$ for every ordinal $\gb\in\bigcup_{z\in a}\dom(h_z)$. Since $\bigcup_{z\in a}[h_z]=0$, there must be a point $z\in a$ such that $O'_0\Vdash\check h_z\subset\eta$. The set $O'=O\cap f_0^{-1}O'_0\cap f^{-1}O_{1z}$ is nonempty and open, and it forces in $P_X$ that $\eta/f_0(\dotxgen)\in [h_z]$ and $[h_z]\subset\tau/f_1(\dotxgen)$.

 The implication (2)$\to$(1) is best proved by a contrapositive. Suppose that (1) fails, as witnessed by some open set $O\subset X$. Let $O'\subset O$ be some nonempty open set whose closure is a subset of $O$, and let $x\in O'$ be a point $P_X$-generic over the ground model. Write $y_0=f_0(x)\in Y_0$ and $y_1=f_1(x)\in Y_1$. Let $C=f_0''(f_1^{-1}\{y_1\}\cap \bar O')$. This is a closed subset of $Y_0$ coded in $V[y_1]$ which contains the point $y_0$. For the failure of (3), it is enough to show that $C$ contains no ground model point. Indeed, if $y\in Y_0$ is a point in the ground model, then $D=f_1''(f_0^{-1}\{y\}\cap \bar O')\subset Y_1$ is a closed subset of $Y_1$ coded in the ground model which has empty interior by the choice of the set $O$; in particular, $D\subset Y_1$ is nowhere dense, and since $y_1\in Y_1$ is a Cohen generic point, $y_1\notin D$ holds. Comparing the definitions of the sets $C$ and $D$, it is obvious that $y_0\notin C$ as required.
\end{proof}

\begin{example}
\label{thetaexample}
Let $b$ be a finite set and let $b=a_0\cup a_1$ be a partition into two sets, each of cardinality at least two. Let $X$ be the closed subset of $\power(\gw)^b$ consisting of those functions $x$ such that $\bigcap_{i\in a}x(i)=0$ and $\bigcup_{i\in a}x(i)=\gw$. Let $Y_0=\power(\gw)^{a_0}$ and $Y_1=\power(\gw)^{a_1}$. Let $f_0\colon X\to Y_0$ and $f_1\colon X\to Y_1$ be the projection functions. Then $f_0, f_1$ are continuous, open, and transcendental over each other.
\end{example}

\begin{proof}
The continuity and openness are left to the reader. To show that $f_1$ is transcendental over $f_0$, let $O\subset X$ be a nonempty relatively open set. Thinning the set $O$ down if necessary, one can find a natural number $k\in\gw$ and sets $c_i\subset k$ for $i\in b$ such that $\bigcup_ic_i=k$ and $\bigcap_ic_i=0$, and $O=\{x\in X\colon \forall i\in b\ x(i)\cap k=c_i\}$. Now let $y_0\in Y_0$ be any point such that $\forall i\in a_0\ y_0(i)\cap k=c_i$ and for some $i\in a_0$ $y_0(i)\subset k$, and for another $i\in a_0$ $\gw\setminus k\subset y_0(i)$. There is such a point because $|a_0|\geq 2$ holds by the assumptions. Now, it is clear that the set $f_1''(f_0^{-1}\{y_0\}\cap O)$ is exactly the open set of all points $y_1\in Y_1$ such that $\forall i\in a_1\ y_1(i)\cap k=c_i\}$.
\end{proof}

\begin{example}
\label{dexample}
Let $b$ be a finite set and let $b=a_0\cup a_1$ be a partition into nonempty sets. Let $X, Y_0, Y_1$ be the closed subsets of $\power(\gw)^b$, $\power(\gw)^{a_0}$, and $\power(\gw)^{a_1}$ consisting of tuples of pairwise disjoint subsets of $\gw$ respectively. Let $f_0\colon X\to Y_0$ and $f_1\colon X\to Y_1$ be the projection functions. Then $f_0, f_1$ are continuous, open, and mutually transcendental  functions.
\end{example}

\begin{proof}
The continuity and openness is left to the reader. For the transcendental part, I will show that $f_1$ is transcendental over $f_0$. Let $O\subset X$ be a relatively open nonempty set. Find finite sets $c_i, d_i\subset\gw$ for each $i\in b$ such that $c_i\cap d_i=0$ and the set $\{\langle z_i\colon i\in b\rangle\in \power(\gw)\colon \forall i\in b\ c_i\subset z_i$ and $d_i\cap z_i=0\}\cap X$ is a nonempty subset of $O$. Note that the sets $c_i$ for $i\in b$ must be pairwise disjoint, and we may arrange the sets $d_i$ so that if $i, j\in b$ are distinct elements then $c_i\subset d_j$. Let $y_0=\langle c_i\colon i\in a_0\rangle$ and let $O_1=\{\langle z_i\colon i\in a_0\rangle\in \power(\gw)\colon \forall i\in a_0\ c_i\subset z_i$ and $d_i\cap z_i=0\}\cap Y_1$; this is a nonempty open subset of $Y_1$. It is clear that for each point $y_1\in O_1$, $\langle y_0, y_1\rangle\in O$ holds and the proof is complete.
\end{proof}

\noindent Another class of examples of transcendental pairs of generic extensions comes from actions of Polish groups with dense diagonal orbits \cite{rosendal:amalgamation}. I am going to need a local variant of this notion which appears to be satisfied in all natural actions with dense diagonal orbits. Recall that if a group $G$ acts on a set $X$, then it also acts coordinatewise on the set $X^n$ for every natural number $n$. 

\begin{definition}
Let $G$ be a Polish group acting continuously on a Polish space $X$. The action has

\begin{enumerate}
\item \emph{dense diagonal orbits} if for every $n\in\gw$ there is a point $\vec x\in X^n$ such that $\{ g\cdot\vec x\colon g\in G\}$ is dense in $X^n$.
\item \emph{locally dense diagonal orbits} if for every open neighborhood $U\subset G$ of the unit and every nonempty open set $O\subset X$ there is a nonempty open set $O'\subset O$ such that for every $n\in\gw$ there is a point $\vec x\in X^n$ such that $\{g\cdot\vec x\colon g\in U\}$ is dense in $(O')^n$.
\end{enumerate}
\end{definition}

\begin{proposition}
\label{diagonalproposition}
Let $G$ be a Polish group acting on a Polish space $X$ with locally dense diagonal orbits. Let $Y\subset G\times X^2$ be the closed set of all triples $\langle g, x_0, x_1\rangle$ such that $g\cdot x_0=x_1$. Let $P_Y$ be its associated Cohen forcing and $\langle \dot g, \dot x_0, \dot x_1\rangle$ its names for the generic triple. $P_Y$ forces the following:

\begin{enumerate}
\item $\dot g$ is $P_G$-generic over $V$;
\item $\langle\dot x_0, \dot x_1\rangle$ is $P_{X^2}$-generic over $V$;
\item the model $V[\dot x_0, \dot x_1]$ is transcendental over $V[\dot g]$.
\end{enumerate}
\end{proposition}

\begin{proof}
For the first item, let $p\in P_Y$ be a condition and $D\subset G$ an open dense set. I must find a stronger condition which forces $\dot g$ into $D$. There are nonempty open neighborhoods $U\subset G$ and $O\subset X$ such that $\langle g, x, g\cdot x\rangle\in p$ whenever $g\in U$ and $x\in O$. Now, just note that the set $D\cap U$ is nonempty; therefore the set of all triples $\langle g, x, g\cdot x\rangle$ where $g\in U\cap D$ and $x\in O$ is a nonempty relatively open subset of $Y$ which forces $\dot g\in D$ as desired.

For the second item, suppose first that $p\in P_Y$ is a condition and $D\subset X^2$ is an open dense set. I must find a stronger condition which forces the pair $\langle\dot x_0, \dot x_1\rangle$ into $D$. There is a point $g\in G$, an open neighborhood $U\subset G$ of the unit, and an open set $O\subset X$ such that $\langle gh, x_0, gh\cdot x_0\rangle\in p$ for all $h\in UU^{-1}$ and $x_0\in O$. Use the dense orbit assumption to thin out the set $O$ if necessary so that for every $n\in\gw$ there is a point $\vec x\in X^n$ such that $\{h\cdot\vec x\colon h\in U\}$ is dense in $O^n$. Since the set $D\subset X^2$ is open dense, there are open sets $P_0\subset O$ and $P_1\subset gO$ such that $P_0\times P_1\subset D$. By the choice of the set $O\subset X$, there must be a point $x\in P_0$ and a point $h\in UU^{-1}$ such that $hx\in g^{-1}P_1$, in other words $ghx\in P_1$. Now the relatively open set of all triples in $p$ such that their second and third coordinates belong to $P_0$ and $P_1$ respectively is nonempty, and it forces $\langle \dot x_0, \dot x_1\rangle\in D$ as required.

For the third item, suppose that $\ga$ is an ordinal, $\tau$ is a $P_{X^2}$-name for an open subset of $2^\ga$ which is forced to contain $V\cap 2^\ga$, and $\eta$ is a $P_G$-name for an element of $2^\ga$. Suppose that $p\in P_Y$ is a condition. One can find an element $h\in G$, an open neighborhood $U\subset G$ of the unit, and a nonempty open set $O\subset X$ such that  $\langle gh, x_0, gh\cdot x_0\rangle\in p$ for all $h\in UU^{-1}$ and $x_0\in O$. Use the dense orbit assumption to thin out the set $O$ if necessary so that for every $n\in\gw$ there is a point $\vec x\in X^n$ such that $\{h\cdot\vec x\colon h\in U\}$ is dense in $O^n$. 

Now, use the initial assumption on the name $\tau$ to find, for each $z\in 2^\ga$, a finite partial map $h_z\colon \ga\to 2$ and a condition $O_{0z}\times O_{1z}\subset O\times gO$ such that $h_z\subset z$ and $O_{0z}\times O_{1z}\Vdash [h_z]\subset\tau$. Use a compactness argument to find a finite set $a\subset 2^\ga$ such that $2^\ga\subset \bigcup_{z\in a}[h_z]$. Now, the dense orbit assumption provides points $x_{0z}\in O_{0z}$ and a point $h\in UU^{-1}$ such that for each $z\in a$, $h\cdot x_{0z}\in g^{-1}O_{1z}$, or in other words $gh\cdot x_{0z}\in O_{1z}$. Let $U'\subset UU^{-1}$ be an open neighborhood such that for all $k\in U'$ and all $z\in a$, $gk\cdot x_{0z}\in O_{1z}$. Shrinking $U'$ if necessary, assume that $gU'$ decides the value of $\eta\restriction\bigcup_{z\in a}\dom(h_z)$. By the choice of the set $a$, there must be a point $z\in a$ such that $gU'\Vdash \check h_z\subset \eta$. Now the relatively open set of all triples in $p$ whose coordinates belong to $gU'$, $O_{0z}$ and $O_{1z}$ respectively is nonempty, and it forces $\eta\in\tau$ as desired.
\end{proof}

\begin{example}
\label{fexample}
Let $Y$ be the closed subset of $S_\infty^4$ consisting of all quadruples $\langle g_0, g_1, g_2, g_3\rangle$ such that $g_0g_1^{-1}g_2g_3^{-1}=1$. The Cohen poset $P_Y$ adds a generic quadruple $\langle\dot g_0, \dot g_1, \dot g_2, \dot g_3\rangle$. It forces $V[\dot g_0, \dot g_2]$ and $V[\dot g_1,\dot g_3]$ to be mutually transcendental $P_{S_\infty^2}$-generic extensions of the ground model.
\end{example}

\begin{proof}
Consider the continuous action of $S_\infty^2$ on $S_\infty$ given by $(h_0, h_2)\cdot h_1=h_0h_1h_2^{-1}$. It has locally dense diagonal orbits: if $U\subset (S_\infty)^2$ is an open neighborhood of the unit and $O\subset S_\infty$ is a nonempty open set, then thinning down one may assume that there is a number $n\in\gw$ such that $U=\{\langle h_0, h_2\rangle\colon h_0\restriction n=h_1\restriction n$ is the identity$\}$ and $O=\{h_1\colon [v]\subset h_1\}$ for some permutation $v$ of $n$. Then the action of $U$ on $O$ is naturally homeomorphic to the whole action of $S_\infty^2$ on $S_\infty$. That action though has dense diagonal orbits because already the conjugation action of $S_\infty$ on $S_\infty$ has them \cite{rosendal:amalgamation}.

Now, to show for example that $P_Y$ forces $V[\dot g_1, \dot g_3]$ to be transcendental over $V[\dot x_0,\dot x_1]$, consider the self-homeomorphism of $S_\infty^4$ which takes inverses of the second and third coordinates. Note that $\langle g_0, g_1, g_2, g_3\rangle\in Y$ iff $(g_0, g_2^{-1})\cdot g_1^{-1}=g_3$ and apply Proposition~\ref{diagonalproposition}.
\end{proof}

\section{Preservation theorems}
\label{preservationsection}

As with all similar notions of perpendicularity of generic extensions, transcendence gives rise to a natural companion: a preservation property for Suslin forcings.

\begin{definition}
Let $P$ be a Suslin forcing. We say that a virtual condition $\bar p$ in $P$ is \emph{transcendentally balanced} if for every pair of mutually transcendental generic extensions $V[G_0], V[G_1]$ inside some ambient forcing extension, and for all conditions $p_0\in V[G_0]$ and $p_1\in V[G_1]$ stronger than $\bar p$, $p_0$ and $p_1$ have a common lower bound.
\end{definition}

\noindent I now state and several preservation theorems for transcendentally balanced extensions of the symmetric Solovay model.

\begin{theorem}
\label{thetatheorem}
Let $\kappa$ be an inaccessible cardinal. In cofinally transcendentally balanced forcing extensions of the symmetric Solovay model derived from $\kappa$, every nonmeager subset of $\power(\gw)$ contains a collection $d$ of cardinality four such that $\bigcap d=0$ and $\bigcup d=\gw$, both modulo finite.
\end{theorem}

\begin{proof}
Let $P$ be a Suslin forcing which is cofinally transcendentally balanced below $\kappa$. Let $W$ be the symmetric Solovay model derived from $\kappa$ and work in the model $W$. Suppose that $p\in P$ is a condition, $\tau$ is a $P$-name, and $p\Vdash\tau\subset \power(\gw)$ is a nonmeager set. I must find a set $d\subset \power(\gw)$ of size four such that $\bigcap d=0$ and $\bigcup d=\gw$, both modulo finite, and a strengthening of the condition $p$ which forces $\check d\subset\tau$.

To this end, let $z\in\cantor$ be a point such that $p, \tau$ are both definable from the parameter $z$ and some parameters in the ground model. Let $V[K]$ be an intermediate forcing extension obtained by a poset of cardinality less than $\kappa$ such that $z\in V[K]$ and $V[K]\models P$ is transcendentally balanced. Work in $V[K]$. Let $\bar p\leq p$ be a transcendentally balanced virtual condition. Let $Q$ be the Cohen poset of nonempty open subsets of $\power(\gw)$, adding a single generic point $\dot z$. There must be a condition $q\in Q$ and a poset $R$ of cardinality smaller than $\kappa$ and an $Q\times R$-name $\gs$ for a condition in $P$ stronger than $\bar p$ such that $q\Vdash_Q R\Vdash\coll(\gw, <\kappa)\Vdash\gs\Vdash_P\dot z\in\tau$. Otherwise, in the model $W$ the condition $\bar p$ would force $\tau$ to be disjoint from the co-meager set of elements of $\power(\gw)$ which are Cohen-generic over $V[K]$, contradicting the initial assumption on $\tau$.

Now, let $X=\{\langle x\rangle\in \power(\gw)^4\colon\bigcup_{i\in 4}x(i)=\gw$ and $\bigcap_{i\in 4}x(i)=0\}$ with the topology inherited from $\power(\gw)^4$. Let $x\in X$ be a point generic over $V[K]$ for the Cohen poset with $X$. By Example~\ref{thetaexample}, $x(0), x(1)$ are mutually Cohen-generic elements of $\power(\gw)$, so are $x(2), x(3)$, and the models $V[K][x(0), x(1)]$ and $V[K][x(2), x(3)]$ are mutually transcendental. Choose finite modifications $z_i$ of $x_i$ such that $z_i\in q$ holds for all $i\in 4$; each of these points is still $Q$-generic over $V[K]$ and meets the condition $q$. Let $H_i\colon i\in 4$ be filters on $R$ mutually generic over the model $V[K][x]$ and let $p_i=\gs/z_i, H_i$. By Proposition~\ref{productproposition}, conclude that the models $V[K][z_0][H_0]$ and $V[K][z_1][H_1]$ are mutually generic extensions of $V[K]$, so are $V[K][z_2][H_2]$ and $V[K][z_3][H_3]$, and the models $V[K][z_0, z_1][H_0, H_1]$ and $V[K][z_2, z_3][H_2, H_3]$ are mutually transcendental extensions of $V[K]$. Now, the balance assumption on the virtual condition $\bar p$, we see that the conditions $p_0, p_1$ have a common lower bound $p_{01}$ in the model $V[K][z_0, z_1][H_0, H_1]$, the conditions $p_2$ and $p_3$ have a common lower bound $p_{23}$ in the model $V[K][z_2, z_3][H_2, H_3]$, and finally the conditions $p_{01}$ and $p_{23}$ have a common lower bound as well. The forcing theorem then shows that such a lower bound then forces in the model $W$ that $\check z_i\in\tau$ holds for all $i\in 4$. The proof is complete.
\end{proof}

\noindent For every number $n\geq 2$ let $\Theta_n$ be the hypergraph of arity $n$ on $\power(\gw)$ consisting of sets $d$ of size $n$ such that $\bigcap d=0$ and $\bigcup d=\gw$, both modulo finite.

\begin{corollary}
\label{thetacorollary}
Let $\kappa$ be an inaccessible cardinal. In cofinally transcendentally balanced forcing extensions of the symmetric Solovay model derived from $\kappa$, the chromatic number of $\Theta_4$ is uncountable.
\end{corollary}

\begin{theorem}
Let $\kappa$ be an inaccessible cardinal. In cofinally transcendentally balanced forcing extensions of the symmetric Solovay model derived from $\kappa$, every nonmeager subset of $S_\infty$ contains a quadruple of distinct points solving the equation $g_0g_1^{-1}g_2g_3^{-1}=1$.
\end{theorem}

\begin{proof}
Let $P$ be a Suslin forcing which is cofinally transcendentally balanced below $\kappa$. Let $W$ be the symmetric Solovay model derived from $\kappa$ and work in the model $W$. Suppose that $p\in P$ is a condition, $\tau$ is a $P$-name, and $p\Vdash\tau\subset S_\infty$ is a nonmeager set. I must find distinct points $z_0, z_1, z_2, z_3\in S_\infty$ such that $z_0z_1^{-1}z_2z_3^{-1}=1$ and a strengthening of the condition $p$ which forces all four of these points into $\tau$.

To this end, let $z\in\cantor$ be a point such that $p, \tau$ are both definable from the parameter $z$ and some parameters in the ground model. Let $V[K]$ be an intermediate forcing extension obtained by a poset of cardinality less than $\kappa$ such that $z\in V[K]$ and $V[K]\models P$ is transcendentally balanced. Work in $V[K]$. Let $\bar p\leq p$ be a transcendentally balanced virtual condition. Let $Q$ be the Cohen poset of nonempty open subsets of $S_\infty$, adding a single generic point $\dot g$. There must be a condition $q\in Q$ and a poset $R$ of cardinality smaller than $\kappa$ and an $Q\times R$-name $\gs$ for a condition in $P$ stronger than $\bar p$ such that $q\Vdash_Q R\Vdash\coll(\gw, <\kappa)\Vdash\gs\Vdash_P\dot g\in\tau$. Otherwise, in the model $W$ the condition $\bar p$ would force $\tau$ to be disjoint from the co-meager set of elements of $S_\infty$ which are Cohen-generic over $V[K]$, contradicting the initial assumption on $\tau$.

Now, let $X=\{x\in S_\infty^4\colon x(0)x(1)^{-1}x(2)x(3)^{-1}=1\}$ with the topology inherited from $S_\infty^4$. Consider the nonempty relatively open set $O\subset X$ given by $O=q^4\cap X$. Note that the set $O$ is indeed nonempty because any constant quadruple in $S_\infty^4$ belongs to $X$. Let $\langle z_i\colon i\in 4\rangle\in O$ be a tuple generic over $V[K]$ for the Cohen poset with $X$. By Example~\ref{fexample}, $z_0, z_2$ are mutually Cohen-generic elements of $S_\infty$ below the condition $q$, so are $z_1, z_3$, and the models $V[K][z_0, z_2]$ and $V[K][z_1, z_3]$ are mutually transcendental. Let $H_i\colon i\in 4$ be filters on $R$ mutually generic over the model $V[K][z_0, z_1, z_2, z_3]$ and let $p_i=\gs/g_i, H_i$. By Proposition~\ref{productproposition}, conclude that the models $V[K][z_0][H_0]$ and $V[K][z_2][H_2]$ are mutually generic extensions of $V[K]$, so are $V[K][z_1][H_1]$ and $V[K][z_3][H_3]$, and the models $V[K][z_0, z_2][H_0, H_2]$ and $V[K][z_1, z_3][H_1, H_3]$ are mutually transcendental extensions of $V[K]$. Now, the balance assumption on the virtual condition $\bar p$, we see that the conditions $p_0, p_2$ have a common lower bound $p_{02}$ in the model $V[K][z_0, z_2][H_0, H_2]$, the conditions $p_1$ and $p_3$ have a common lower bound $p_{13}$ in the model $V[K][z_1, z_3][H_1, H_3]$, and finally the conditions $p_{02}$ and $p_{13}$ have a common lower bound as well. The forcing theorem then shows that such a lower bound then forces in the model $W$ that $\check z_i\in\tau$ holds for all $i\in 4$. The proof is complete.
\end{proof}

\begin{corollary}
\label{deltacorollary}
Let $\kappa$ be an inaccessible cardinal. In cofinally transcendentally balanced forcing extensions of the symmetric Solovay model derived from $\kappa$, the chromatic number of the hypergraph on $S_\infty$ consisting of solutions to the equation $g_0g_1^{-1}g_2g_3^{-1}=1$ is uncountable.
\end{corollary}

\noindent Certain consistency results require amalgamation diagrams with multiple forcing extensions. The following definitions and a theorem show one such possibility.

\begin{definition}
Let $m>n$ be natural numbers. A poset $P$ is $m, n$-\emph{centered} if for every set $a\subset P$ of cardinality $m$, if every subset of $a$ of cardinality $\leq n$ has a common lower bound, then $a$ has a common lower bound.
\end{definition}

\noindent Note that in the case of a Suslin poset $P$, the $m, n$-centeredness of $P$ is a $\mathbf{\Pi}^1_2$-statement and therefore absolute among all forcing extensions by a Shoenfield absoluteness argument.

\begin{theorem}
\label{gammatheorem}
Let $\kappa$ be an inaccessible cardinal and $n\geq 2$ be a number. In $n+1, n$-centered, cofinally transcendentally balanced forcing extensions of the symmetric Solovay model derived from $\kappa$, every nonmeager subset of $\power(\gw)$ contains $n+1$ many sets which modulo finite form a partition of $\gw$.
\end{theorem}

\begin{proof}
Let $P$ be a Suslin forcing which is $n+1, n$-centered and cofinally transcendentally balanced below $\kappa$. Let $W$ be the symmetric Solovay model derived from $\kappa$ and work in the model $W$. Suppose that $p\in P$ is a condition, $\tau$ is a $P$-name, and $p\Vdash\tau\subset\power(\gw)$ is a nonmeager set. I must find a collection $\{a_i\colon i\in n+1\}$ which is modulo finite a partition of $\gw$ and a condition stronger than $p$ which forces every element of this collection into $\tau$.

To this end, let $z\in\cantor$ be a point such that $p, \tau$ are both definable from the parameter $z$ and some parameters in the ground model. Let $V[K]$ be an intermediate forcing extension obtained by a poset of cardinality less than $\kappa$ such that $z\in V[K]$ and $V[K]\models P$ is transcendentally balanced. Work in $V[K]$. Let $\bar p\leq p$ be a transcendentally balanced virtual condition. Let $Q$ be the Cohen poset of nonempty open subsets of $\power(\gw)$, adding a single generic point $\dot a$. There must be a condition $q\in Q$ and a poset $R$ of cardinality smaller than $\kappa$ and an $Q\times R$-name $\gs$ for a condition in $P$ stronger than $\bar p$ such that $q\Vdash_Q R\Vdash\coll(\gw, <\kappa)\Vdash\gs\Vdash_P\dot a\in\tau$. Otherwise, in the model $W$ the condition $\bar p$ would force $\tau$ to be disjoint from the co-meager set of elements of $\power(\gw)$ which are Cohen-generic over $V[K]$, contradicting the initial assumption on $\tau$.

Let $X$ be the closed subset of $\power(\gw)^{n+1}$ consisting of tuples of sets which form a partition of $\gw$ and consider the poset $P_X$ of relatively open subsets of $X$, adding a generic partition $\dot x$ of $\gw$ into $n+1$ many sets. Move to the model $W$ and find an $n+1$-tuple $x$ which is generic over the model $V[K]$ for the poset $P_X$. Let $H_i\subset R$ for $i\in n+1$ be a collection of filters mutually generic over the model $V[K][x]$. For each $i\in n+1$, make a finite adjustment to $x(i)$ so that the resulting set $a_i\subset\gw$ meets the condition $q$; note that $a_i$ is $Q$-generic over $V[K]$. Let $p_i=\gs/a_i, H_i$. Since $p_i\Vdash a_i\in\tau$, it will be enough to show that the conditions $p_i$ for $i\in n+1$ have a common lower bound in the poset $P$.

By the centeredness assumption on the poset $P$, it is enough to show that for every set $b\subset n+1$ of cardinality $n$, the conditions $p_i$ for $i\in b$ have a common lower bound in $P$.
To do this, return to $V[K]$ and consider the closed subset $Y_b$ of $\power(\gw)^b$ consisting of pairwise disjoint sets. It is easy to check that the projection map from $X$ to $Y_b$ is open.
Thus, the restriction $x\restriction b$ is generic over the model $V[K]$ for the poset $P_{Y_b}$ of nonempty relatively open subsets of $Y$ by \cite[Proposition 3.1.1]{z:geometric}. 
Now, let $\langle i_j\colon j\in n\rangle$ be an enumeration of the set $b$. By recursion on $j\in n$ build conditions $q_j\in V[K][x(i_k), H_{i_k}\colon k\leq j]$ which are a common lower bound of all $p_{i_k}$ for $k\leq j$ respectively. To start, let $q_0=p_{i_0}$. For the recursion step, suppose that the condition $q_j$ has been found. Note that the models $V[K][x(i_k)\colon k\in j]$ and $V[K][x(i_j)]$ are mutually transcendental by Example~\ref{dexample}. It follows that also models $V[K][x(i_k), H_{i_k}\colon k\in j]$ and $V[K][x(i_j)][H_{i_j}]$ are mutually transcendental over $V[K]$ by Proposition~\ref{productproposition}. By the balance assumption on $\bar p$, it must be the case that the conditions $q_j$ and $p_{i_{j+1}}$ have a common lower bound, and one such common lower bound $q_{j+1}$ must be in the model $V[K][x(i_k), H_{i_k}\colon k\in j+1]$ by a Mostowski absoluteness argument. In the end, the condition $q_{n-1}$ is a common lower bound of the conditions $p_i$ for $i\in b$ as required.
\end{proof}

\begin{corollary}
\label{gammacorollary}
Let $n\geq 2$ be a natural number and $\Gamma_{n+1}$ be the hypergraph on $\power(\gw)$ of $n+1$-tuples which form a modulo finite partition of $\gw$. In $n+1, n$-centered, cofinally transcendentally balanced forcing extensions of the symmetric Solovay model derived from $\kappa$, the chromatic number of $\Gamma_{n+1}$ is uncountable.
\end{corollary}

\section{Examples II}
\label{exampleIIsection}

The whole enterprise in the previous sections would be pointless if there were no substantial transcendentally balanced posets. In this section, I will produce or point out a number of examples in this direction. At first, I consider posets or classes of posets known from previous work.

\begin{proposition}
Every placid Suslin poset is transcendentally balanced.
\end{proposition}

\noindent This class of examples is very broad: it includes among others posets adding a Hamel basis for a Polish space over a countable field, posets adding maximal acyclic subsets to Borel graphs, or posets adding a selector to pinned Borel equivalence relations classifiable by countable structures.

\begin{proof}
Recall \cite[Definition 9.3.1]{z:geometric} that a poset $P$ is placid if below every condition $p\in P$ there is a virtual balanced condition $\bar p\leq p$ which is placid: whenever $V[G_0]$ and $V[G_1]$ are generic extensions such that $V[G_0]\cap V[G_1]=V$ and $p_0\in V[G_0]$ and $p_1\in V[G_1]$ are conditions stronger than $\bar p$, then $p_0, p_1$ are compatible. Now, if $V[G_0], V[G_1]$ are mutually transcendental extensions of the ground model, then $V[G_0]\cap V[G_1]=V$ by Proposition~\ref{iproposition}, and therefore a placid virtual condition also transcendentally balanced. The proposition follows.
\end{proof}

\begin{proposition}
Let $X$ be a $K_\gs$ Polish field with a countable subfield $F$. The poset adding a transcendence basis to $X$ over $F$ is transcendentally balanced.
\end{proposition}

\begin{proof}
Reviewing the proof of \cite[Theorem 6.3.9]{z:geometric} it becomes clear that the only feature of mutually generic extensions $V[G_0]$ and $V[G_1]$ there is that if $p$ is a multivariate polynomial with coefficients in $F$, $\vec x_0\in X\cap V[G_0]$ and $\vec x_1\in X\cap V[G_1]$ are tuples such that $p(\vec x_0, \vec x_1)=0$, then there are tuples $\vec x_0', \vec x_1'$ in the ground model such that
$p(\vec x'_0, \vec x_1)=p(\vec x_0, \vec x_1')=0$. However, this is satisfied for mutually transcendental extensions $V[G_0], V[G_1]$ as well by Corollary~\ref{jcorollary}. This completes the proof.
\end{proof}

\begin{proposition}
Let $E$ be an equivalence relation on a Polish space of one of the following types:

\begin{enumerate}
\item $E$ is $K_\gs$;
\item for some sequence $\langle Y_n, d_n\colon n\in\gw\rangle$ of countable metric spaces, $E$ is the equivalence relation on $X=\prod_nY_n$ connecting points $x_0, x_1$ if the distances $d_n(x_0(n), x_1(n))$ tend to zero as $n$ tends to infinity.
\end{enumerate}

\noindent  The poset adding a selector to $E$ is transcendentally balanced.
\end{proposition}

\begin{proof}
Note that the equivalence relation $E$ is pinned (\cite[Chapter 17]{kanovei:book}, but it follows directly from Corollary~\ref{kscorollary} or Proposition~\ref{metricproposition}) and therefore \cite[Theorem 6.4.5]{z:geometric} applies. The only feature of mutually generic extensions $V[G_0]$ and $V[G_1]$ in the proof of the balance of $P$ is that every $E$-class represented both in $V[G_0]$ and $V[G_1]$ is represented in $V$. However, this feature holds true for mutually transcendental extensions by Corollary~\ref{kscorollary} or Proposition~\ref{metricproposition}.
\end{proof}

\noindent Now it is time to produce transcendentally balanced posets for some new and more difficult tasks. I will only look at coloring posets for hypergraphs of a certain type.

\begin{definition}
Let $X$ be a Polish space, and $\Gamma$ a hypergraph on $X$. $\Gamma$ is \emph{redundant} if for every set $a\subset X$, the set $\{x\in X\colon a\cup \{x\}\in\Gamma\}$ is countable.
\end{definition}

\begin{example}
The hypergraph $\Gamma$ on $\mathbb{R}$ of arity $3$ consisting of solutions to the equation $x^3+y^3+z^3-3xyz=0$ is redundant.
\end{example}

\begin{example}
The hypergraph $\Gamma$ on $\mathbb{R}^2$ consisting of vertices of equilateral triangles is redundant. A similar hypergraph on $\mathbb{R}^3$ is not redundant.
\end{example}

\begin{example}
Let $n\geq 2$ be a number. The hypergraph $\Gamma_n$ on $\power(\gw)$ consisting of $n$-tuples which modulo finite partition $\gw$ is redundant.
\end{example}

\begin{example}
Let $G$ be a Polish group with a countable dense subset $d\subset G$. Let $n\geq 2$ be a natural number. The hypergraph $\Sigma(G, n)$ consisting of all $n$-tuples whose product belongs to $d$ is redundant. Note that if $G$ is not abelian, then the product depends on the order of the elements, so one must say ``the product of all elements in some order belongs to $d$''.
\end{example}

\begin{example}
Let $G$ be a Polish group and $n\geq 2$ be a natural number. The hypergraph $\Theta(G, n)$ of all $2n$-tuples whose alternating product $g_0g_1^{-1}g_2g_3^{-1}\dots$ in some order is equal to $1$ is redundant.
\end{example}

\begin{theorem}
\label{bigtheorem}
Let $X$ be a $K_\gs$-Polish space and $\Gamma$ an $F_\gs$-hypergraph on $X$ of arity three (or four). Then there is a coloring forcing $P_\Gamma$ such that

\begin{enumerate}
\item $P_\Gamma$ is $\gs$-closed Suslin poset;
\item for every number $n\in\gw$, the poset $P_\Gamma$ is $n, 3$-centered (or $n, 4$-centered, respectively);
\item the union of the generic filter is forced to be a total $\Gamma$-coloring on $X$ with countable range;
\item if the Continuum Hypothesis holds, then the poset $P_\Gamma$ is transcendentally balanced.
\end{enumerate}
\end{theorem}

\noindent There are many hypergraphs (such as the hypergraph of rectangles on $\mathbb{R}^2$ \cite{komjath:three}) of arity four for which the Continuum Hypothesis assumption in the last item is necessary. I do not know if it is possible to eliminate the CH assumption for arity three in general. Arities higher than four present obstacles that I know how to overcome only in the case of algebraic hypergraphs ???

It is now possible to prove the theorems from the introduction. For Theorem~\ref{1theorem}, observe that for a $K_\gs$-Polish group $G$, the hypergraph $\Delta(G)$ is redundant and $F_\gs$ of arity four. Theorem~\ref{bigtheorem} applies to provide a $\gs$-closed Suslin coloring forcing $P$ which is transcendentally balanced under CH. It follows from Corollary~\ref{deltacorollary} that in the $P$-extension of the choiceless Solovay model, the chromatic number of $\Delta(G)$ is countable while that of $\Delta(S_\infty)$ is not. For Theorem~\ref{2theorem}, observe that the graph $\Gamma_3$ is $F_\gs$ on the compact space $\power(\gw)$, and it is redundant. Theorem~\ref{bigtheorem} provides a $4, 3$-centered coloring poset $P$ which is transcendentally balanced under CH. Corollary~\ref{gammacorollary} then shows that in the $P$-extension of the choiceless Solovay model, $\Gamma_3$ is countably chromatic while $\Gamma_4$ is not. The proof of Theorem~\ref{3theorem} is similar. For Theorem~\ref{4theorem}, note that Theorem~\ref{bigtheorem} provides a coloring poset $P$ which is transcendentally balanced under CH. Corollary~\ref{thetacorollary} then shows that in the $P$-extension of the choiceless Solovay model, $\Gamma_4$ is countably chromatic while $\Theta_4$ is not.  

\subsection{Construction in arity three}

The construction of the coloring poset shares many similarities in both arities, but it is slightly easier in arity three. The following concept is shared.

\begin{definition}
Let $\Gamma$ be a redundant hypergraph on a Polish space $X$.

\begin{enumerate}
\item A set $b\subset X$ is $\Gamma$-\emph{closed} if for every set $a\subset b$ the countable set  $\{x\in X\colon a\cup \{x\}\in\Gamma\}$ is a subset of $b$;
\item if a set $b\subset X$ is $\Gamma$-closed, define the equivalence relation $E(b, \Gamma)$ on $X\setminus b$ as the smallest equivalence containing all pairs $\{x_0, x_1\}$ such that for some set $a\subset b$, $a\cup\{x_0, x_1\}\in\Gamma$.
\end{enumerate}
\end{definition}

Note that if the set $b$ is countable then all classes of the relation $E(b, \Gamma)$ are countable. If in addition the hypergraph $\Gamma$ is Borel, then so is the relation $E(b, \Gamma)$.
In both arities, the definition of the poset uses as a parameter a Borel ideal $\mathcal{I}$ on $\gw$ which contains all singletons and is not generated by countably many sets. Beyond these requirements the choice of $\mathcal{I}$ appears to be immaterial.

\begin{definition}
\label{pgdefinition}
Let $X$ be a $K_\gs$ Polish space and $\Gamma$ be a redundant $F_\gs$ hypergraph of arity three on $X$. The \emph{coloring poset} $P_\Gamma$ consists of all partial $\Gamma$ colorings $p\colon X\to\gw$ whose domain is a countable $\Gamma$-closed subset of $X$. The ordering is defined by $p_1\leq p_0$ if $p_0\subset p_1$ and for every $E(\Gamma, \dom(p_0))$-class $a\subset\dom(p_1)$, $p_1''a\in\mathcal{I}$.
\end{definition}

\begin{proposition}
$\leq$ is a transitive $\gs$-closd relation on $P_\Gamma$.
\end{proposition}

\begin{proof}
For the transitivity, suppose that $p_2\leq p_1\leq p_0$ are conditions in $P_\Gamma$, and work to show that $p_2\leq p_0$ holds. Clearly, $p_2\subset p_0$. Let $C\subset X$ be a $E(p_0, \Gamma)$-equivalence class. By the $\Gamma$-closure of $\dom(p_1)$, $C$ is either a subset of $p_1$, or it is disjoint from $\dom(p_1)$ and then it is a subset of a single $E(\dom(p_1), \Gamma)$-class. In the former case, $p_2''C=p_1''C\in\mathcal{I}$ as $p_1\leq p_0$ holds. In the latter case, $p_2''C\in\mathcal{I}$ as $p_2\leq p_1$ holds. This concludes the proof of $p_2\leq p_0$ and the transitivity of the $\leq$ relation.

 For the $\gs$-closure, let $\langle p_n\colon n\in\gw\rangle$ be a descending chain of conditions in $P_\Gamma$. Let $q=\bigcup_np_n$ and argue that $q\in P_\Gamma$ is a lower bound of the chain. It is clear that $\dom(q)$ is $\Gamma$-closed and $q$ is a $\Gamma$-coloring. Now, fix $n\in\gw$ and work to show that $q\leq p_n$ holds. Clearly, $p_n\subset q$. Now let $C\subset\dom(q)$ be an $E(\dom(p_n),\Gamma)$-equivalence class. By the $\Gamma$-closures of the domains of all conditions mentioned, there must be a number $m\geq n$ such that $C\subset \dom(p_{k+1}\setminus p_k)$ and then $q''C=p_{k+1}''C\in\mathcal{I}$ follows from $p_{k+1}\leq p_n$.

\end{proof}

\noindent The main point in the definition of the coloring poset is that there is a precise and generous criterion for compatibility of conditions in it.

\begin{proposition}
\label{cocoproposition}
Let $a\subset P_\Gamma$ be a finite set. The following are equivalent:

\begin{enumerate}
\item $a$ has a common lower bound;
\item for every point $z\in X$, $a$ has a common lower bound whose domain contains $z$;
\item $\bigcup a$ is a function, a $\Gamma$-coloring, and for distinct conditions $p_0, p_1\in a$, $p_0''C\in\mathcal{I}$ where $C\subset X$ is any $E(\dom(p_1), \Gamma)$-class.
\end{enumerate}
\end{proposition}

\begin{proof}
It is clear that (2) implies (1) and (1) implies (3). The significant direction is (3) implies (1). Suppose that (3) holds and $x\in X$ is an arbitrary point. Let $M$ be a countable elementary submodel of a large enough structure containing $\Gamma$, $a$, and $z$ in particular. Let $e=M\cap X\setminus\bigcup_{p\in a}\dom(p)$. For each point $x\in e$, let $d_x=\bigcup\{p_0''C\colon p_0\in a$ and for some $p_1\in a$ distinct from $p_0$, $C$ is the $E(\dom(p_1), \Gamma)$-class of $x\}$; by (3), this is a set in $\mathcal{I}$.  Let $d\subset\gw$ be a set in the ideal $\mathcal{I}$ which is not modulo finite covered by any set $d_x$ for $x\in e$. Let $q\colon X\cap M\to\gw$ be any function such that $\bigcup a\subset q$, $q\restriction e$ is an injection, and for each $x\in e$, $q(x)\in d\setminus d_x$. I claim that $q\in P$ is a lower bound of all conditions in the set $a$.

To show that $q\in P$ holds, the elementarity of the model $M$ shows that $\dom(q)$ is $\Gamma$-closed. Now, let $x_0, x_1, x_2\in M\cap X$ be distinct points forming a $\Gamma$-hyperedge, and work to argue that they do not all receive the same value in the coloring $q$.
If they all belong to $\bigcup_{p\in a}\dom(p)$ then this follows from (3). If at least two of them belong to $e$, then this follows from the fact that $q\restriction e$ is an injection. If at least two of them belong to $\dom(p)$ for the same condition $p\in a$, then by the $\Gamma$-closure of $\dom(p)$, all three belong to $\dom(p)$ and they cannot receive the same color as $p$ is a $\Gamma$-coloring. The only remaining configuration is that (after reindexing, if necessary) there are distinct conditions $p_0, p_1\in a$ such that $x_0\in\dom(p_0)\setminus\dom(p_1)$, $x_1\in\dom(p_1)\setminus\dom(p_0)$, and $x_2\in e$ holds. In this case, $x_0, x_2$ are $E(\dom(p_1), \Gamma)$-related and $x_1, x_2$ are $E(\dom(p_0), \Gamma)$-related by the definitions. Since $q(x_2)\notin d_{x_2}$, it follows that $x_2$ receives a color distinct from both the color of $x_0$ and $x_1$.

Now, let $p\in a$ be a condition and argue that $q\leq p$ holds. It is clear that $p\subset q$. Now, let $C\subset X$ be a $E(\dom(p), \Gamma)$-class and argue that $q''C\in\mathcal{I}$. This, however, follows from the fact that $q''C\subset d\cup\bigcup\{p_1''C\colon p_1\in a, p_1\neq p_0\}$ as all sets in the union on the right belong to the ideal $\mathcal{I}$ by (3).
\end{proof}

\begin{corollary}
The poset $P_\Gamma$ is Suslin.
\end{corollary}

\begin{proof}
For the Suslinity of $\leq$, it is clear that $\leq$ is a Borel relation. The Borelness of its compatibility relation follows directly from Proposition~\ref{cocoproposition}.
\end{proof}

\begin{corollary}
The poset $P_\Gamma$ is $n, 3$-centered.
\end{corollary}

\begin{proof}
If item (3) of Proposition~\ref{cocoproposition} fails for a finite set $a$, then it fails for a subset of it of cardinality at most three.
\end{proof}

\begin{corollary}
$P_\Gamma$  forces the union of the generic filter to be a total $\Gamma$-coloring of $X$.
\end{corollary}

\begin{proof}
 Proposition~\ref{cocoproposition} applied to sets $a$ of cardinality one shows that for every point $z\in X$, the collection of conditions $p$ containing $z$ in their domain is dense in $P_\Gamma$. 
\end{proof}

\subsection{Construction in arity four}

\begin{definition}
\label{pg4definition}
Let $X$ be a $K_\gs$ Polish space and $\Gamma$ be a redundant $F_\gs$ hypergraph of arity four on $X$. The \emph{coloring poset} $P_\Gamma$ consists of all partial $\Gamma$ colorings $p\colon X\to\gw$ whose domain is a countable $\Gamma$-closed subset of $X$. The ordering is defined by $p_1\leq p_0$ if 

\begin{enumerate}
\item $p_0\subset p_1$;
\item (tight) for every $E(\Gamma, \dom(p_0))$-class $C\subset\dom(p_1)$, $p_1\restriction a$ is injective and $p_1''C\in\mathcal{I}$;
\item (slick) for every $x\in\dom(q)\setminus\dom(p)$, the set $c(x, p, q)=\{i\in\gw\colon$ for some points $x_0\in\dom(p)$ and $x_1, x_2\in\dom(q)\setminus\dom(p)$ such that $p(x_0)=q(x_1)=q(x_2)$ and $\{x, x_0, x_1, x_2\}\in\Gamma\}$ belongs to the ideal $\mathcal{I}$.
\end{enumerate}
\end{definition}

\noindent The definition of $P_\Gamma$ is very similar to arity three, except for the injective part of the tight item and the mysterious slick item. The injectivity part is used to ensure the balance of the poset $P_\Gamma$; some version of it is necessary in the configuration discussion in the proof of Proposition~\ref{4balance}. The slick item causes a lot of grief below, and its only function is to assert that the poset is $n, 4$-centered.

\begin{proposition}
The relation $\leq$ is transitive and $\gs$-closed.
\end{proposition}

\begin{proof}
To see the transitivity, suppose that $p_2\leq p_1\leq p_0$ are conditions in $P_\Gamma$, and work to show that $p_2\leq p_0$ holds. Clearly, $p_2\subset p_0$. To verify the tight item, let $C\subset X$ be a $E(p_0, \Gamma)$-equivalence class. By the $\Gamma$-closure of $\dom(p_1)$, $C$ is either a subset of $p_1$, or it is disjoint from $\dom(p_1)$ and then it is a subset of a single $E(\dom(p_1), \Gamma)$-class. In the former case, $p_2\restriction C=p_1\restriction C$ is an injection and $p_2''C=p_1''C\in\mathcal{I}$ as $p_1\leq p_0$ holds. In the latter case, $p_2\restriction C$ is an injection and $p_2''C\in\mathcal{I}$ as $p_2\leq p_1$ holds. To verify the slick item, suppose that $x\in\dom(p_2\setminus p_0)$ is a point. The discussion breaks into cases.

 Suppose first that $x\in\dom(p_1)$. Let $x_0\in\dom(p_0)$ and $x_1, x_2\in\dom(p_2)$ be points of the same color forming a $\Gamma$-hyperedge with $x$. It is impossible for both $x_1, x_2$ to belong to $\dom(p_2\setminus p_1)$ since then they would be $E(\dom(p_1), \Gamma)$-related and of distinct colors by the tight item of $p_2\leq p_1$. It is also impossible that exactly one of the points $x_1, x_2$ belongs to $\dom(p_2\setminus p_1)$ by the $\Gamma$-closure of $\dom(p_1)$. Thus, both points $x_1, x_2$ belong to $\dom(p_1)$, so $c(x, p_0, p_2)=c(x, p_0, p_1)\in\mathcal{I}$ by the slick item of $p_1\leq p_0$.

Suppose now that $x\in\dom(p_2\setminus p_1)$. Let $x_0\in\dom(p_0)$ and $x_1, x_2\in\dom(p_2)$ be points of the same color forming a $\Gamma$-hyperedge with $x$. If one of the points $x_1, x_2$ belongs to $\dom(p_1)$, then the other is $E(\dom(p_1), \Gamma)$-related to $x$. Consequently, $c(x, p_0, p_2)\subset c(x, p_1, p_2)\cup p_2''C$ where $C$ is the $E(\dom(p_1), \Gamma)$-class of $x$. So $c(x, p_0, p_2)\in\mathcal{I}$ by a combination of slick and tight items of $p_2\leq p_1$. This concludes the proof of $p_2\leq p_0$ and the transitivity of the $\leq$ relation.

 For the $\gs$-closure, let $\langle p_n\colon n\in\gw\rangle$ be a descending chain of conditions in $P_\Gamma$. Let $q=\bigcup_np_n$ and argue that $q\in P_\Gamma$ is a lower bound of the chain. This is left to the reader.
\end{proof}

\noindent The key point of the definition of the poset is that there is a precise and generous characterization of compatibility of its conditions.

\begin{proposition}
\label{ciciproposition}
Let $a\subset P_\Gamma$ be a finite set. The following are equivalent:

\begin{enumerate}
\item $a$ has a common lower bound;
\item for every point $z\in X$, $a$ has a common lower bound whose domain contains $z$;
\item $\bigcup a$ is a function, a $\Gamma$-coloring, and for any three conditions $p_0, p_1, p_2\in a$, $p_0$ distinct from the other two, and for every point $x\in X\setminus\dom(p_0)$, writing $C$ for the $E(\dom(p_0), \Gamma)$-class of $x$, the function $p_1\cup p_2\restriction C$ is injective with range in $\mathcal{I}$, and the set $c(x, p_0, p_1\cup p_2)$ belongs to the ideal $\mathcal{I}$.
\end{enumerate}
\end{proposition}

\begin{proof}
Clearly, (2) implies (1). (1) implies (3) by the tight and slick items of the definition of the ordering. To show that (3) implies (2), let $a\subset P_\Gamma$ be a finite set satisfying (3), and let $z\in X$ be any point. To find the lower bound required in (3), let $M$ be a countable elementary submodel of a large structure containing $a$ and $z$ in particular. Write $e=M\cap X\setminus\bigcup_{p\in a}\dom(p)$. For each point $x\in e$, write $d_{x0}=\bigcup\{p_0''C\colon p_0\in a$ and for some $p_1\in a$ distinct from $p_0$, $C$ is the $E(\dom(p_1), \Gamma)$-class of $x\}$; by (3), this is a set in $\mathcal{I}$. Also, write $d_{x1}=\bigcup\{c(x, p_0, p_1\cup p_2)\colon p_0\in a$ and $p_1, p_2\in a$ are conditions distinct from $p_0\}$; by (3) again, this is a set in $\mathcal{I}$. Let $d\subset\gw$ be a set in the ideal $\mathcal{I}$ which is not modulo finite covered by any set $d_{x0}\cup d_{x1}$ for $x\in e$. Let $q\colon X\cap M\to\gw$ be any function such that $\bigcup a\subset q$, $q\restriction e$ is an injection, and for each $x\in e$, $q(x)\in d\setminus (d_{x0}\cup d_{x1})$. I claim that $q\in P$ is a lower bound of all conditions in the set $a$.

To show that $q\in P$ holds, the elementarity of the model $M$ shows that $\dom(q)$ is $\Gamma$-closed. To show that $q$ is a $\Gamma$-coloring, suppose that $\gamma=\{x_i\colon i\in 4\}$ is a $\Gamma$-hyperedge consisting of points in $M\cap X$. The discussion splits into several cases. If no points of $\gamma$ fall into the set $e$, then $\gamma$ is not monochromatic since $\bigcup a$ is a $\Gamma$-coloring. If more than one point of $\gamma$ belongs to $e$ then $\gamma$ is not monochromatic since $q\restriction e$ is an injection. The remaining case is that exactly one point of $\gamma$, say $x_3$, belongs to $e$. Now, no condition in $a$ can contain all three of the remaining points of $\gamma$ by the $\Gamma$-closure of the domains of the conditions. This leads to two subcases. Either, there is a condition $p_1$ which contains exactly two points of $\gamma$, say $x_1, x_2$. Let $p_0\in a$ be a condition which contains the remaining point $x_0\in\gamma$. Observe that $x_0$ and $x_3$ are $E(\dom(p_0), \Gamma)$-related, so $q(x_0)\neq q(x_3)$ since $q(x_3)\notin d_{x_30}$. Or, each condition of $a$ contains at most one point of $\gamma$. Let $p_0\in a$ be a condition containing $x_0$. Then, either the set $x_0, x_1, x_2$ is not monochromatic in $\bigcup a$, or else $q(x_3)$ is distinct from its monochromatic color as $q(x_3)\notin d_{x_31}$.

To show that $q$ is a common lower bound of the set $a$, let $p\in a$ be any condition. To show that $q\leq p$ holds, I must verify the tight item and the slick item of Definition~\ref{pg4definition}. For the tight item, let $C\subset\dom(p_1)$ be a $E(\dom(p_0), \Gamma)$-class of some point $x\in \dom(q\setminus p_0)$. Then $q\restriction (C\setminus e)$ is an injective function with range in $\mathcal{I}$ by (3), and $q\restriction (C\cap e)$ is an injection with range in $\mathcal{I}$ which in addition does not use any values of $q\restriction (C\setminus e)$ by the choice of $q$. This confirms the tight item. For the slick item, use (3) in addition with the fact that the range of $q\restriction e$ belongs to the ideal $\mathcal{I}$.
\end{proof}

\begin{corollary}
$P_\Gamma$ is a Suslin partial order.
\end{corollary}

\begin{proof}
It is clear that the relation $\leq$ is Borel. The compatibility of conditions in $P_\Gamma$ is Borel as well as per Proposition~\ref{ciciproposition} and the observation that the universal quantification over $x$ in it can be restricted to the $\Gamma$-closure of $\dom(\bigcup(a))$.
\end{proof}

\begin{corollary}
$P_\Gamma$ is $n, 4$-centered for any $n\in\gw$. 
\end{corollary}

\begin{proof}
If item (3) of Proposition~\ref{ciciproposition} fails for a finite set, then it fails for a subset of it of cardinality at most four.
\end{proof}

\begin{corollary}
$P_\Gamma$ forces the union of the generic filter to be a total $\Gamma$-coloring of $X$.
\end{corollary}

\begin{proof}
Proposition~\ref{ciciproposition} applied to sets $a$ of cardinality one shows that for every point $z\in X$, the collection of conditions $p$ containing $z$ in their domain is dense in $P_\Gamma$.
\end{proof}

\subsection{The balance proof}

The posets constructed in the previous subsections are transcendentally balanced under the Continuum Hypothesis. The proof depends on the following technical claim.

\begin{proposition}
\label{equiproposition}
Let $X$ be a $K_\gs$ Polish space and $\Gamma$ a redundant $F_\gs$ hypergraph of arity $n\geq 2$ on $X$. Let $V[G_0], V[G_1]$ be mutually transcendental generic extensions of the ground model. Then on $X\cap V[G_0]$, $E(V\cap X, \Gamma)= E(V[G_1]\cap X, \Gamma)$.
\end{proposition}

\begin{proof}
The left-to-right inclusion is obvious as increasing the set $b$ increases the equivalence relation $E(b, \Gamma)$. The right-to-left inclusion is the heart of the matter. Suppose that $x, x'\in X\cap V[G_0]$ are two points in $X\cap V[G_0]$ which are $E(X\cap V[G_1], \Gamma)$ equivalent. Then there must be a number $m\in\gw$, points $x_i\in X$ for $i\leq m$ and sets $y_i\in [X\cap V[G_1]]^{n-2}$ for $i<m$ such that $x=x_0$, $x'=x_m$, and $\forall i<m\ y_i\cup \{x_i, x_{i+1}\}\in\Gamma$. The tuple $\langle x_i\colon i\leq m, y_i\colon i<m\rangle$ will be called a \emph{walk} from $x$ to $x'$.

Let $K\subset X$ be a compact set coded in the ground model containing all points mentioned in the walk. Let $d$ be a complete metric on $X$ and let $\eps>0$ be a positive rational such that for any two points mentioned in the walk, if they are distinct then they have $d$-distance at least $\eps$. Let $\Delta\subset\Gamma$ be a ground model coded compact set such that all hyperedges in the walk belong to $\Delta$. 

Now, consider the space $Y=\{\langle z_i\colon i\in m\rangle\colon z_i\in [K]^{n-2}$ and distinct points in each $z_i$ have a distance at least $\eps\}$; this is a compact subspace of $([K]^{n-2})^m$ in the ground model. Consider the set $C\subset Y$ consisting of tuples $\langle z_i\colon i\in m\rangle$ which can serve in a walk from $x$ to $x'$ which uses only points in $K$, whose hyperedges belong to $\Delta$, and in which any two distinct points have distance at least $\eps$. The set $C\subset Y$ is compact, as it is a projection of a compact set of walks. The set $C$ is coded in $V[G_0]$, and the sequence $\langle y_i\colon i\in m\rangle\in V[G_1]$ belongs to it. By the mutual transcendence of the models $V[G_0]$ and $V[G_1]$, the set $C$ contains a ground model element. A review of definitions reveals that this  means that $x, x'$ are $E(X\cap V, \Gamma)$-related.
\end{proof}

\begin{proposition}
\label{4balance}
Let $\Gamma$ be a redundant $F_\gs$ hypergraph of arity three or four on a $K_\gs$ Polish space $X$. In the poset $P_\Gamma$, 

\begin{enumerate}
\item for every total $\Gamma$-coloring $c\colon X\to \gw\times\gw$, the pair $\langle\coll(\gw, X), \check c\rangle$ is transcendentally balanced;
\item for every balanced pair $\langle Q, \tau\rangle$ there is a total coloring $c\colon X\to \gw\times\gw$ such that the pairs $\langle Q, \tau\rangle$ and $\langle\coll(\gw, X), \check c\rangle$ are equivalent;
\item distinct total $\Gamma$-colorings provide inequivalent balanced pairs.
\end{enumerate}
\end{proposition}

\begin{proof}
For the first item of the theorem, suppose that $c\colon X\to\gw\times\gw$ is a total coloring. Clearly, $\coll(\gw, X)\Vdash\check c\in P_\Gamma$ holds. Now suppose that $V[G_0], V[G_1]$ are mutually transcendental generic extensions of the ground model and $p_0\in V[G_0]$, $p_1\in V[G_1]$ are conditions stronger than $c$. I must prove that the conditions are compatible. 

\noindent\textbf{Arity three.} Use Proposition~\ref{cocoproposition} and verify its item (3). It is clear that $p_0\cap p_1$ is a function, since the domains of $p_0$ and $p_1$ intersect in $X\cap V$, and on that set both $p_0, p_1$ are equal to $c$.
It is also clear that $p_0\cup p_1$ is a $\Gamma$-coloring, as any hyperedge in its domain has to have two elements either in $\dom(p_0)$ or in $\dom(p_1)$. By the $\Gamma$-closure of these two sets it follows that all three elements of the hyperedge must belong to one of them, so the hyperedge is not monochromatic as both $p_0, p_1$ are $\Gamma$-colorings. Finally, note that $E(\dom(p_0), \Gamma)\restriction \dom(p_1)$ is equal to $E(X\cap V, \Gamma)\restriction\dom(p_1)$ by Proposition~\ref{equiproposition}. Thus, for every $E(\dom(p_0), \Gamma)$-class $C$, $C\cap \dom(p_1)$ is a subset of a single $E(X\cap V, \Gamma)$-class and $p_1''C\in\mathcal{I}$ follows from $p_1\leq C$.

\noindent\textbf{Arity four.} Use Proposition~\ref{ciciproposition} and verify its item (3). It is clear that $p_0\cap p_1$ is a function, since the domains of $p_0$ and $p_1$ intersect in $X\cap V$, and on that set both $p_0, p_1$ are equal to $c$.

To see that $p_0\cup p_1$ is a $\Gamma$-coloring, assume that $\gamma\in\Gamma$ is a hyperedge. If one of the conditions $p_0, p_1$ contains at least three elements of $\gamma$, then it contains all of them by the $\Gamma$-closure of its domain, so $\gamma$ is not monochromatic as both $p_0, p_1$ are $\Gamma$-colorings. The only other case is that each $p_0, p_1$ contain exactly two elements of $\gamma$. Then, the two points of $\gamma$ in $\dom(p_0)$ do not belong to $V$, and they are $E(\dom(p_1), \Gamma)$-equivalent. By Proposition~\ref{equiproposition}, they are also $E(X\cap V, \Gamma)$-equivalent. It follows from the tight item of $p_0\leq c$ that the two points receive distinct colors by $p_1$.

To see that for every point $x$ in $X$, the function $p_1''C$ is an injection with range in $\mathcal{I}$, where $C$ is the $E(\dom(p_1), \Gamma)$-equivalence class of $x$, apply  Proposition~\ref{equiproposition} again to see that $C\cap\dom(p_0)$ is in fact a single $E(X\cap V, \Gamma)$-class and apply the tight item of $p_0\leq c$. To see that for every point $x\in X$ the set $c(x, p_0, p_1)$ belongs to $\mathcal{I}$, note that it is a subset of $p_0''C$ where $C$ is the $E(\dom(p_1), \Gamma)$-equivalence class of $x$.

For the second item of the theorem, let $\langle Q, \tau\rangle$ be a balanced pair. Strengthening $\tau$ if necessary, we may assume that $Q\Vdash X\cap V\subset\dom(\tau)$. A balance argument then shows that for every ground model point $x\in X$ there is a pair $c(x)\in\gw\times\gw$ such that $Q\Vdash \tau(\check x)=c(x)$. I claim that the $\Gamma$-coloring $c$ works as in (2). It will be enough to show that $Q\Vdash\tau\leq\check c$. If this failed, then there must be a condition $q\in Q$ which forces the failure of $\tau\leq c$. Let $G_0, G_1\subset Q$ be mutually generic filters containing the condition $q$ and let $p_0=\tau/G_0$ and $p_1=\tau/G_1$. The two conditions $p_0, p_1$ should be compatible in $P_\Gamma$, but the contradictory assumption together with Propositions~\ref{cocoproposition} or~\ref{ciciproposition} shows that they are not.

The third item of the theorem is immediate. 
\end{proof}

\begin{proposition}
Under CH, the poset $P_\Gamma$ is transcendentally balanced.
\end{proposition}

\begin{proof}
Suppose that the continuuum hypothesis holds and let $p\in P_\Gamma$ be a condition. I must find a total coloring $c\colon X\to\gw\times\gw$ which is stronger than $p$ in the sense of the ordering on the poset $P_\Gamma$. To do this, use the CH assumption to find an enumeration $\langle x_\ga\colon\ga\in\gw_1\rangle$ of the space $X$. By recursion on $\ga\in\gw_1$ build a decreasing sequence of conditions in $P_\Gamma$ so that $p=p_0$, $p_{\ga+1}\leq p_\ga$ is a condition containing $x_\ga$ in its domain, and $p_\ga=\bigcup_{\gb\in\ga}p_\gb$ for limit ordinals $\ga$. In the end, let $c=\bigcup_\ga p_\ga$; it is easy to check that $c\leq p$ as required.
\end{proof}

\bibliographystyle{plain}
\bibliography{odkazy,zapletal}

\end{document}